\documentclass[preprint,11pt]{elsarticle}

\usepackage{amssymb}
\usepackage{amsmath,amsthm,amsfonts,amssymb,latexsym,mathrsfs,color,hyperref}

\newtheorem{theorem}{Theorem}

\newtheorem{proposition}[theorem]{Proposition}

\newtheorem{lemma}[theorem]{Lemma}

\newtheorem{example}[theorem]{Example}

\newcommand{\sech}{\ensuremath{\mathrm{sech\ }}}

\newcommand{\sgn}{{\rm sgn\,}}
\newcommand{\rz}{{\rm RZ}}

\newcommand{\seps}{\prec}
\newcommand{\R}{{\mathbb R}}

\newcommand{\arxiv}[1]{\href{http://arxiv.org/abs/#1}{\texttt{arXiv:#1}}}
\journal{}

\begin{document}

\begin{frontmatter}



\title{The half interlacing property among the types $A,B$ and $D$ Eulerian polynomials}

\author[focal]{Shi-Mei Ma}
\ead{shimeimapapers@163.com}
\address[focal]{School of Mathematics and Statistics, Shandong University of Technology, Zibo, Shandong 255000, P.R. China}
\begin{abstract}
A famous result in the theory of combinatorial polynomials is the real-rootedness of the type $D$ Eulerian polynomial $D_n(x)$, which was originally conjectured by Brenti in 1994.
By constructing a set of compatible polynomials over $s$-inversion sequences, Savage and Visontai proved this conjecture in 2013. Using matrices preserving interlacing properties of nonnegative polynomial sequences,
Br\"anden also established the real-rootedness of $D_n(x)$.         
Combining Hermite-Biehler theorem and a result of Borcea and Br\"and\'en on Hurwitz stability, 
Yang and Zhang gave another proof of the real-rootedness of $D_n(x)$. 
By constructing half Eulerian polynomials of type $D$, Hyatt reproved Brenti's conjecture. 
As originally suggested by Brenti in 1994, it is possible that the 
real-rootedness of $D_n(x)$ may be established by using a more precise knowledge of the location of zeros of the types $A$ and $B$ Eulerian polynomials. 
In this paper, we add more details to the first proof of the real-rootedness of $D_n(x)$ that was provided by the author in 2012, which yields 
the half interlacing property among the types $A,B$ and $D$ Eulerian polynomials.
\end{abstract}


\end{frontmatter}

\section{Introduction}
Eulerian polynomials have been objects of study for centuries~\cite{Petersen15}.
They are not only of interest in combinatorics~\cite{Brenti94,Chow08,Liu07}, 
but also of significance in geometry~\cite{Branden22,Savage15} and algebra~\cite{Brenti94,Stembridge94}.
The Eulerian polynomials of types $A$ and $B$ can be respectively defined by
\begin{equation}\label{Anx-recu}
\begin{split}
A_{n+1}(x)&=(nx+1)A_{n}(x)+x(1-x)\frac{\mathrm{d}}{\mathrm{d}x}A_{n}(x),\\
B_{n+1}(x)&=(2nx+x+1)B_{n}(x)+2x(1-x)\frac{\mathrm{d}}{\mathrm{d}x}B_{n}(x),
\end{split}
\end{equation}
with $A_0(x)=B_0(x)=1$ (see~\cite{Brenti94,Chow08}). 
In~\cite[Lemma 9.1]{Stembridge94}, Stembridge found that the type $D$ Eulerian polynomials $D_n(x)$ can be computed by the following identity:
\begin{equation}\label{Dnx-recu}
D_n(x)=B_n(x)-n2^{n-1}xA_{n-1}(x) \quad {\text{ for $n\geqslant 2$}}.
\end{equation}
Clearly, $\deg A_n(x)=n-1$, $\deg B_n(x)=n$ and $\deg D_n(x)=n$. 
It is well known that the types $A$, $B$ and $D$ Eulerian polynomials share several similar properties, see~\cite{Brenti94,Petersen15,Savage15}.

The real-rootedness of $A_n(x)$ was first established by Frobenius in 1910, see~\cite{Petersen15} for instance. 
In 1994, Brenti showed that $B_n(x)$ has only simple negative zeros~\cite[Corollary~3.7]{Brenti94}, and conjectured that $D_n(x)$ has only real zeros~\cite[Conjecture~5.1]{Brenti94}, which became a long-standing conjecture.
Below are the polynomials $D_n(x)$ for $2\leqslant n\leqslant 4$: 
\begin{align*}
  D_2(x)&=1+2x+x^2,\\
  D_3(x)&=1+11x+11x^2+x^3,\\
  D_4(x)& =1+44x+102x^2+44x^3+x^4,\\
  D_5(x)&=1+157x+802x^2+802x^3+157x^4+x^5.
\end{align*}

In~\cite{Ma12}, using derivative polynomials, the author gave the first proof of the real-rootedness of $D_n(x)$, and in this paper we shall add more details 
to the proofs of~\cite[Eq.~(3.2),~Eq.~(3.3)]{Ma12}. 
Subsequently, by constructing a set of compatible polynomials over $s$-inversion sequences,
Savage and Visontai~\cite{Savage15} gave a proof of the real-rootedness of $D_n(x)$. Using matrices preserving interlacing properties of nonnegative polynomial sequences,
Br\"anden~\cite{Branden15} gave another proof of the real-rootedness of $D_n(x)$.  
Combining Hermite-Biehler theorem and a result of Borcea and Br\"and\'en on Hurwitz stability, 
Yang and Zhang~\cite{Yang15} gave another proof of the real-rootedness of $D_n(x)$. 
Inspired by Savage and Visontai's proof, using the mutual interlacing method, 
Yang and Zhang~\cite{Yang17} established the real-rootedness of $D_n(x,q)$ for any $q>0$, where $q$ marks the negative numbers of even signed permutations. 
By constructing half Eulerian polynomials of type $D$, Hyatt~\cite{Hyatt16} reproved Brenti's conjecture. 

As Brenti~\cite{Brenti94} wrote: ``It is possible that~\cite[Conjecture~5.1]{Brenti94} may be solved using a more precise knowledge of 
the location of zeros of the types $A$ and $B$ Eulerian polynomials''. 
Based on~\eqref{Dnx-recu}, we are now ready to give an affirmative answer to Brenti's suggestion.

\begin{theorem}\label{Thm01}
For any $n\geqslant 2$, 
the type $D$ Eulerian polynomial has only simple negative zeros and there is a half interlacing property among $A_{n-1}(x)$, $B_n(x)$ and $D_n(x)$. More precisely, 
we write 
\begin{equation}
\begin{aligned}
A_{n-1}(x)&=\prod_{i=1}^{n-2}(x-a_i),~
B_{n}(x)=\prod_{i=1}^{n}(x-b_i),~
D_{n}(x)=\prod_{i=1}^{n}(x-d_i),
\end{aligned}
\end{equation}
where $a_1<a_2<\cdots \cdots <a_{n-2}<0,~b_1<b_2<\cdots \cdots <b_{n}<0$ and $d_1<d_2<\cdots \cdots <d_{n}<0$.
When $n=2m$, we have $d_m<-1$, $-1<d_{m+1}<0$ and
\begin{equation}\label{zeros01}
\begin{aligned}
b_1<d_1<a_1<b_2<d_2<a_2<&\cdots <a_{m-1}<b_m<d_m<d_{m+1}<b_{m+1}<a_m<\\
&d_{m+2}<b_{m+2}<a_{m+1}<\cdots<a_{2m-2}<d_{2m}<b_{2m};
\end{aligned}
\end{equation}
When $n=2m+1$, we have $a_m=b_{m+1}=d_{m+1}=1$ and
\begin{equation}\label{zeros02}
\begin{aligned}
b_1<d_1<a_1<b_2<d_2<a_2<&\cdots <a_{m-1}<b_m<d_m<d_{m+1}=b_{m+1}=a_m<\\
&d_{m+2}<b_{m+2}<a_{m+1}<\cdots<a_{2m-1}<d_{2m+1}<b_{2m+1}.
\end{aligned}
\end{equation}
In other words, the zeros in the interval $(-\infty,-1)$ and $(-1,0)$ are respectively interlacing. 
\end{theorem}

We say that two polynomials are {\it half interlacing} if their zeros satisfy~\eqref{zeros01} or~\eqref{zeros02}, 
since there are two zeros of a polynomial appear consecutively in the middle of all zeros. 
\begin{example}
When $n=7$, we have 
\begin{align*}
A_6(x)&=1+57x+302x^2+302x^3+57x^4+x^5,\\
B_7(x)&=1 + 2179x + 60657x^2 + 259723x^3 + 259723x^4 + 60657x^5 + 2179x^6 + x^7\\
D_7(x)&=1+1731x+ 35121x^2+ 124427x^3+ 124427x^4+ 35121x^5+ 1731x^6+x^7.
\end{align*}
Then their zeros satisfy~\eqref{zeros02}, since
\begin{align*}
a_1&=-51.2184,~a_2=-4.54193,~a_3=-1,~a_4=-0.220171,~a_5=-0.0195242;\\
b_1&=-2150.85,~b_2=-23.1557,~b_3=-3.67364,~b_4=-1,\\
b_5&=-0.27221,~b_6=-0.0431859,~b_7=-0.000464931;\\
d_1&=-1710.51,~d_2=-16.2985,~d_3=-2.7683,~d_4=-1,\\
d_5&=-0.361232,~d_6=-0.0613554,~d_7=-0.000584621,
\end{align*}
\end{example}

\section{Proof of Theorem~\ref{Thm01}}
It is well known that
$$\sum_{n=0}^{\infty}A_n(-1)\frac{x^n}{n!}=1+\tanh (x),~\sum_{n=0}^{\infty}B_n(-1)\frac{x^n}{n!}=\sech (2x),$$
see~\cite{Hirzebruch08} for instance.
Intuitively, as discussed in~\cite{Han25,Hoffman95}, one should tackle both of the hyperbolic tangent and hyperbolic secant at the same time.
The derivative polynomials for hyperbolic tangent and hyperbolic secant obey this principle.
We now define
\begin{equation*}\label{derivapoly-2}
\frac{\mathrm{d}^n}{\mathrm{d}\theta^n}\tanh \theta={P}_n(\tanh \theta)\quad {\text and}\quad
\frac{\mathrm{d}^n}{\mathrm{d}\theta^n}\sech \theta=\sech\theta \cdot{Q}_n(\tanh \theta).
\end{equation*}
By the chain rule, it is routine to check that
 \begin{equation}\label{Differential07}
	\left\{
	\begin{array}{l}
\widetilde{P}_{n+1}(x)=(1-x^2)\frac{\mathrm{d}}{\mathrm{d}x}\widetilde{P}_n(x),~\widetilde{P}_0(x)=x\\
\widetilde{Q}_{n+1}(x)=(1-x^2)\frac{\mathrm{d}}{\mathrm{d}x}\widetilde{Q}_n(x)-x\widetilde{Q}_n(x),~\widetilde{Q}_0(x)=1.
	\end{array}
	\right.
	\end{equation}
For $n\geqslant 1$, we define 
 \begin{equation}
\begin{aligned}
\widetilde{A}_n(x)&=(x-1)(x+1)^{n}A_{n}\left(\frac{x-1}{x+1}\right),\\
\widetilde{B}_n(x)&=\left(\frac{x+1}{2}\right)^nB_n\left(\frac{x-1}{x+1}\right),\\
\widetilde{D}_n(x)&=\left(\frac{x+1}{2}\right)^nD_n\left(\frac{x-1}{x+1}\right).
\end{aligned}
\end{equation}
Clearly, $D_n(x)\in\rz(-\infty,0)$ if and only if $\widetilde{D}_n(x)\in\rz(-1,1)$. 
Substituting these two expressions
into~\eqref{Anx-recu} and simplifying, we obtain
 \begin{equation}\label{Anx-recu02}
\begin{aligned}
\widetilde{A}_{n+1}(x)&=(x^2-1)\frac{\mathrm{d}}{\mathrm{d}x}\widetilde{A}_n(x),~\widetilde{A}_0(x)=x\\
\widetilde{B}_{n+1}(x)&=(x^2-1)\frac{\mathrm{d}}{\mathrm{d}x}\widetilde{B}_n(x)+x\widetilde{B}_n(x),~\widetilde{B}_0(x)=1.
\end{aligned}
	\end{equation}
Below are the first few of these polynomials:
$$\widetilde{A}_1(x)=x^2-1,~\widetilde{A}_2(x)=2x^3-2x,~\widetilde{A}_3(x)=6x^4-8x^2+2;$$
$$\widetilde{B}_1(x)=x,~\widetilde{B}_2(x)=2x^2-1,~\widetilde{B}_3(x)=6x^3-5x.$$
By induction, it is easy to verify that $\widetilde{A}_{n}(x)$ and $\widetilde{B}_n(x)$ have the following expressions:
\begin{equation*}\label{any}
\widetilde{A}_n(x)=\sum_{k=0}^{\lfloor(n+1)/2\rfloor}(-1)^kp(n,n-2k+1)x^{n-2k+1},~
\widetilde{B}_n(x)=\sum_{k=0}^{\lfloor{n/2}\rfloor}(-1)^kq(n,n-2k)x^{n-2k},
\end{equation*}
where $p(n,n+1)=q(n,n)=n!$. 
Compare~\eqref{Differential07} and~\eqref{Anx-recu02}, we see that
\begin{equation}\label{anx}
\widetilde{A}_n(x)=(-1)^n\widetilde{P}_{n}(x)\quad {\text and}\quad \widetilde{B}_n(x)=(-1)^n\widetilde{Q}_{n}(x),
\end{equation}
see~\cite[Theorem~5,~Theorem~6]{Franssens07} and~\cite[Theorem~10]{Han25} for equivalent formulas.
An equivalent formula of~\eqref{Dnx-recu} is given as follows:
\begin{proposition}\label{recu-dnx}
For $n\geqslant 2$, we have
\begin{equation}\label{dnxpnxqnx}
\widetilde{D}_n(x)=\widetilde{B}_n(x)-\frac{n}{2}\widetilde{A}_{n-1}(x).
\end{equation}
\end{proposition}

For $n\geqslant 3$, Chow~\cite[Corollary 6.10]{Chow08} obtained that
\begin{equation}\label{chow}
\sgn D_n(-1)=\begin{cases}
0 & \text{if $n$ is odd},\\
(-1)^{\frac{n}{2}} & \text{if $n$ is even}.
\end{cases}
\end{equation}
It follows from~\eqref{Anx-recu02} and~\eqref{dnxpnxqnx} that $\widetilde{D}_n(-1)=\widetilde{B}_n(-1)=(-1)^n$ for $n\geqslant 2$. 
The first few terms of $\widetilde{D}_n(x)$ can be computed directly as follows:
\begin{align*}
  \widetilde{D}_2(x)&=x^2,~
 \widetilde{D}_3(x)=3x^3-2x,~
  \widetilde{D}_4(x)=12x^4-12x^2+1,\\
  \widetilde{D}_5(x)& =60x^5-80x^3+21x,~
  \widetilde{D}_6(x) =360x^6-600x^4+254x^2-13.
\end{align*}

Let $\rz$ denote the set of real polynomials with only real zeros.
Denote by $\rz(I)$ the set of such polynomials all
whose zeros are in the interval $I$. Suppose that $f,F\in\rz$. Let $\{s_i\}$ and $\{r_j\}$ be all zeros of $F$ and $f$ in nonincreasing
order respectively. Following~\cite{Liu07}, we say that $F$ {\it interleaves} $f$, denoted by
$f\preceq F$, if $\deg f\leqslant\deg F\leqslant \deg f+1$ and
\begin{equation}\label{sep}
s_1\geqslant r_1\geqslant s_2\geqslant r_2\geqslant s_3\geqslant r_3\geqslant\cdots.
\end{equation}
If no equality sign occurs in~\eqref{sep}, then we say that $F$ {\it strictly interleaves} $f$.
Let $f\seps F$ denote $F$ strictly interleaves $f$.

Combining~\eqref{anx} and~\cite[Proposition 6.5,~Theorem 8.6]{Hetyei08}, we have 
the following result.
\begin{lemma}[{\cite[Theorem 8.6]{Hetyei08}}]\label{Hetyei}
For $n\geqslant 2$, the zeros $-1<u_1<u_2<\cdots<u_{n-1}<1$ of $\widetilde{A}_{n}(x)$ and the zeros 
$t_1<t_2<\cdots<t_n$ of $\widetilde{B}_{n}(x)$ are all real, belong to the interval $[-1,1]$, and are interlaced
$-1<t_1<u_1<t_2<\cdots <u_{n-1}<t_n<1$.
Moreover, $\widetilde{A}_{n-1}(x)\preceq\widetilde{A}_{n}(x)$ and
$\widetilde{B}_{n-1}(x)\prec\widetilde{B}_{n}(x)$.
\end{lemma}

\begin{theorem}\label{thm02}
For $n\geqslant 1$, let
\begin{equation}\label{eqAABB}
\begin{aligned}
\widetilde{A}_{2n-1}(x)&=(2n-1)!\prod_{i=1}^n(x-s_i)(x+s_i),~
\widetilde{A}_{2n}(x)=(2n)!x\prod_{i=1}^n(x-a_i)(x+a_i),\\
\widetilde{B}_{2n}(x)&=(2n)!\prod_{j=1}^n(x-r_j)(x+r_j),~
\widetilde{B}_{2n+1}(x)=(2n+1)!x\prod_{j=1}^{n}(x-b_j)(x+b_j).
\end{aligned}
\end{equation}
Then we have
\begin{equation}\label{zeros-1}
1=s_1> r_1>s_2> r_2> \cdots >r_{n-1}>s_n>r_n>0,
\end{equation}
\begin{equation}\label{zeros-2}
1=a_1> b_1>a_2> b_2> \cdots >b_{n-1}>a_n>b_n>0.
\end{equation}
\end{theorem}
Theorem~\ref{thm02} is a key ingredient in our first proof of Theorem~\ref{Thm01}, but the proof of it was omitted in~\cite{Ma12}. 
In the sequel, we shall give a detailed proof of Theorem~\ref{thm02}.

Let $\sgn$ denote the sign function defined on $\R$ by
\begin{equation*}
\sgn x=\begin{cases}
1 & \text{if $x>0$},\\
0 & \text{if $x=0$},\\
-1 & \text{if $x<0$}.
\end{cases}
\end{equation*}
Following Wagner~\cite{Wagner92}, a real polynomial is said to be {\it standard} if either it is identically
zero or its leading coefficient is positive. As usual, let $\frac{\mathrm{d}}{\mathrm{d}x}f(x)=f'(x)$.

First we need to give an inequality involving polynomials with zeros satisfy~\eqref{zeros-1}.
\begin{lemma}\label{lemma-key}
Let $f(x)=c_1\prod_{i=1}^n(x-\alpha_i)$ and $g(x)=c_2\prod_{i=1}^n(x-\beta_i)$ be two standard real polynomials.
Assume that 
\begin{equation}\label{inequality}
1=\alpha_1>\beta_1>\alpha_2>\beta_2>\cdots >\alpha_n>\beta_n>0.
\end{equation}
 Let $R(x)=g(x)/f(x)$. Then for $\gamma \in (0,1)$ and $f(\gamma)\neq 0$, we have
\begin{equation}
R(\gamma)+2(\gamma-1)R'(\gamma)>0.
\end{equation}
\end{lemma}
\begin{proof}
Since $\deg f(x)=\deg g(x)=n$, the partial fraction expansion of $R(x)$ has the form
\begin{equation}\label{Rx}
R(x)=\frac{c_2}{c_1}+\sum_{i=1}^n\frac{\lambda_i}{x-\alpha_i}.
\end{equation}
where $\lambda_i=\frac{g(\alpha_i)}{f'(\alpha_i)}$. It follows from~\eqref{inequality} that $\sgn f'(\alpha_i)=\sgn g(\alpha_i)=(-1)^{i-1}$ 
and so $\lambda_i>0$ for $1\leqslant i\leqslant n$. Differentiating both sides of~\eqref{Rx} with respect to $x$, we obtain
$$R'(x)=-\sum_{i=1}^n\frac{\lambda_i}{(x-\alpha_i)^2}.$$
Therefore, we get
\begin{align*}
R(\gamma)+2(\gamma-1)R'(\gamma)&=\frac{c_2}{c_1}+\sum_{i=1}^n\frac{\lambda_i}{\gamma-\alpha_i}-2(\gamma-1)\sum_{i=1}^n\frac{\lambda_i}{(\gamma-\alpha_i)^2}\\
&=\frac{c_2}{c_1}+\sum_{i=1}^n\frac{\lambda_i(2-\gamma-\alpha_i)}{(\gamma-\alpha_i)^2}>0,
\end{align*}
as desired. This completes the proof.
\end{proof}

\noindent{\bf Proof of Theorem~\ref{thm02}:}
\begin{proof}
It is routine to verify that the result holds for $n=1,2,3$. We proceed by induction. In the following discussion, we always 
set $y=x^2$.
Let $\widetilde{A}_{2n-1}(x)=f_n(y)$ and $\widetilde{B}_{2n}(x)=g_n(y)$, 

Assume that~\eqref{zeros-1} holds. It should be noted that~\eqref{zeros-1} is equivalent to
\begin{equation}\label{zeros-01}
1=s_1^2> r_1^2>s_2^2> r_2^2> \cdots >r_{n-1}^2>s_n^2>r_n^2>0.
\end{equation}
Using~\eqref{Anx-recu02}, we get 
\begin{align*}
\widetilde{A}_{2n}(x)=2x(y-1)f_n'(y),~~
\widetilde{B}_{2n+1}(x)&=x\left(g_n(y)+2(y-1)g_n'(y)\right).
\end{align*}

Let $F(y)=(y-1)f_n'(y)$ and $G(y)=g_n(y)+2(y-1)g_n'(y)$. By~\eqref{eqAABB}, we see that
the zeros of $F(y)$ are $a_1^2>a_2^2>\cdots>a_n^2$ and that of $G(y)$
are $b_1^2>b_2^2>\cdots>b_n^2$. 
From~\eqref{eqAABB}, we see that the zeros of $f_n(y)$ are $s_1^2>s_2^2>\cdots>s_n^2$.
By Rolle's Theorem, one has $f_n'(y)\prec f_n(y)$. Hence
$$f_n'(a_i^2)=0,~s_{i-1}^2>a_{i}^2>s_{i}^2$$ for $2\leqslant i \leqslant n$, which can also be deduced by Lemma~\ref{Hetyei}.

Let $R_n(y)={g_n(y)}/{f_n(y)}$. Then we have
\begin{align*}
G(a_i^2)&=g_n(a_i^2)+2(a_i^2-1)g_n'(a_i^2
)\\
&=R_n(a_i^2)f_n(a_i^2)+2(a_i^2-1)f_n(a_i^2)R_n'(a_i^2)\\
&=f_n(a_i^2)\left(R_n(a_i^2)+2(a_i^2-1)R_n'(a_i^2)\right).
\end{align*}
By Lemma~\ref{lemma-key}, we find that $\sgn G(a_i^2)=\sgn f_n(a_i^2)=(-1)^{i-1}$ for $2\leqslant i \leqslant n$.
Clearly, $G(1)=g_n(1)>0$. Recall that $\widetilde{B}_{2n}(x)=g_n(y)$. So we have 
$$\frac{g_n'(y)}{g_n(y)}=\sum_{i=1}^n\frac{1}{y-r_i^2}.$$
Note in particular that $\frac{g_n'(0)}{g_n(0)}=-\sum_{i=1}^n\frac{1}{r_i^2}$. Therefore, we obtain
$$G(0)=g_n(0)-2g_n'(0)=g_n(0)\left(1+2\sum_{i=1}^n\frac{1}{r_i^2}\right),$$
which yields $\sgn G(0)=\sgn g_n(0)=(-1)^n$. So $G(y)$ has a zero in each of the intervals 
$(0,a_n^2),~(a_2^2,a_3^2),\ldots,~(a_3^2,a_2^2),~(a_2^2,1)$.
i.e., let $b_1^2>b_2^2>\cdots>b_n^2$ be the desired zeros, then 
\begin{equation}\label{zeros04}
1=a_1^2>b_1^2>a_2^2>b_2^2>a_3^2>\cdots>b_{n-1}^2>a_n^2>b_n^2>0,
\end{equation}
which is in agreement with~\eqref{zeros-2}. 

In the same way, assume that~\eqref{zeros-2} holds, we shall show that~\eqref{zeros-1} also holds for $n+1$.
Let $\widetilde{A}_{2n}(x)=xp_n(y)$ and $\widetilde{B}_{2n+1}(x)=xq_n(y)$. The zeros $a_1^2>a_2^2>\cdots>a_n^2$ of $p_n(y)$ and 
the zeros $b_1^2>b_2^2>\cdots>b_n^2$ of $q_n(y)$ satisfy~\eqref{zeros04}.
It follows from~\eqref{Anx-recu02} that 
\begin{align*}
\widetilde{A}_{2n+1}(x)&=(x^2-1)\widetilde{A}_{2n}'(x)=(y-1)\left(p_n(y)+2yp_n'(y)\right),\\
\widetilde{B}_{2n+2}(x)&=(x^2-1)\frac{\mathrm{d}}{\mathrm{d}x}\widetilde{B}_{2n+1}(x)+x\widetilde{B}_{2n+1}(x)=(2y-1)q_n(y)+2y(y-1)q_n'(y).
\end{align*}

Note that $\sgn \left(p_n(a_i^2)+2a_i^2p_n'(a_i^2)\right)=\sgn \left(p_n'(a_i^2)\right)=(-1)^{i-1}$ for $1\leqslant i\leqslant n$, and 
$\sgn \left(p_n(0))\right)=(-1)^n$.
So we find that $p_n(y)+2yp_n'(y)$ has one zero in each of the interval $$(0,a_n^2),~(a_n^2,a_{n-1}^2),\ldots,(a_3^2,a_2^2),(a_2^2,a_1^2).$$
Let $\lambda_1>\lambda_2>\cdots>\lambda_n$ be the zeros of $p_n(y)+2yp_n'(y)$. Then 
\begin{equation}\label{zeros05}
a_1^2>\lambda_1>a_2^2>\lambda_2>\cdots >a_{n-1}^2>\lambda_{n-1}>a_n^2>\lambda_n>0.
\end{equation}

Since $p_n(\lambda_i)+2\lambda_ip_n'(\lambda_i)=0$, it follows that 
\begin{equation}\label{zeros10}
p_n'(\lambda_i)=-\frac{p_n(\lambda_i)}{2\lambda_i}.
\end{equation}
Now define $M(y)=(2y-1)q_n(y)+2y(y-1)q_n'(y)$ and $N(y)=\frac{q_n(y)}{p_n(y)}$. 
Note that 
\begin{align*}
M(\lambda_i)&=(2\lambda_i-1)q_n(\lambda_i)+2\lambda_i(\lambda_i-1)q_n'(\lambda_i)\\
&=(2\lambda_i-1)p_n(\lambda_i)N(\lambda_i)+2\lambda_i(\lambda_i-1)\left(N'(\lambda_i)p_n(\lambda_i)+N(\lambda_i)p_n'(\lambda_i)\right).
\end{align*}
Using~\eqref{zeros10}, we arrive at 
\begin{align*}
M(\lambda_i)&=\lambda_ip_n(\lambda_i)N(\lambda_i)+2\lambda_i(\lambda_i-1)N'(\lambda_i)p_n(\lambda_i)\\
&=\lambda_ip_n(\lambda_i)\left(N(\lambda_i)+2(\lambda_i-1)N'(\lambda_i)\right).
\end{align*}
By Lemma~\ref{lemma-key}, we obtain $\sgn M(\lambda_i)=\sgn \lambda_ip_n(\lambda_i)=\sgn p_n(\lambda_i)=(-1)^i$ for $1\leqslant i\leqslant n$.
It is clear that $\sgn M(0)=\sgn (-q_n(0))=(-1)^{n+1}$ and $\sgn M(1)=\sgn q_n(1)=1$. 
So we find that $M(y)$ has one zero in each of the $n+1$ intervals $(0,\lambda_n),~(\lambda_n,\lambda_{n-1}),\ldots,(\lambda_2,\lambda_1),(\lambda_1,1)$.
Let $\mu_1>\mu_2>\cdots>\mu_{n+1}$ be the zeros of $M(y)$. Then 
\begin{equation}\label{zeros06}
1>\mu_1>\lambda_1>\mu_2>\lambda_2>\cdots >\mu_{n-1}>\lambda_{n-1}>\mu_n>\lambda_n>\mu_{n+1}>0.
\end{equation}
Set $s_1=1$, $s_{i+1}=\sqrt{\lambda_i}$ and $r_j=\sqrt{\mu_j}$ for $1\leqslant i\leqslant n$ and $1\leqslant j\leqslant n+1$. 
We immediately find that~\eqref{zeros-1} holds for $n+1$. This completes the proof.
\end{proof}

Following the notation of Theorem~\eqref{thm02}, we now ready to prove the main result of this paper.

\noindent{\bf A proof of Theorem~\ref{Thm01}:}
\begin{proof}
From~\eqref{dnxpnxqnx}, we get $$\widetilde{D}_{2n}(x)=\widetilde{B}_{2n}(x)-n\widetilde{A}_{2n-1}(x).$$
Let $F(x)=\prod_{i=1}^n(x-s_i)$ and $f(x)=\prod_{j=1}^n(x-r_j)$.
Then $$\widetilde{D}_{2n}(x)=(2n-1)!n(-1)^n\left\{2f(x)f(-x)-F(x)F(-x)\right\}.$$
Note that $\sgn \widetilde{D}_{2n}(s_{j+1})=(-1)^{j}$ and $\sgn \widetilde{D}_{2n}(r_j)=(-1)^{j+1}$, where $1\leqslant j\leqslant n-1$. 
Therefore, $\widetilde{D}_{2n}(x)$ has precisely one zero in each of $2n-2$ intervals $(s_{j+1},r_j)$ and $(-r_j,-s_{j+1})$.
Note that $\sgn \widetilde{D}_{2n}(r_n)=(-1)^{n-1}$ and $\sgn \widetilde{D}_{2n}(-r_n)=(-1)^{n+1}$.
It follows from~\eqref{chow} that $\sgn \widetilde{D}_{2n}(0)=(-1)^n$.
Therefore, $\widetilde{D}_{2n}(x)$ has precisely one zero in each of the intervals $(-r_n,0)$ and $(0,r_n)$.
Thus $\widetilde{D}_{2n}(x)\in\rz(-1,1)$.

Similarly, by~\eqref{dnxpnxqnx}, we get
$$\widetilde{D}_{2n+1}(x)=\widetilde{B}_{2n+1}(x)-\frac{2n+1}{2}\widetilde{A}_{2n}(x).$$
Let $G(x)=\prod_{i=1}^n(x-a_i)$ and $g(x)=\prod_{j=1}^n(x-b_j)$. Then
$$\widetilde{D}_{2n+1}(x)=(2n+1)!(-1)^nx\left\{g(x)g(-x)-\frac{1}{2}G(x)G(-x)\right\}.$$
Note that $\sgn \widetilde{D}_{2n+1}(a_{j+1})=(-1)^{j}$ and $\sgn \widetilde{D}_{2n+1}(b_j)=(-1)^{j+1}$, where $1\leqslant j\leqslant n-1$.
Therefore, $\widetilde{D}_{2n+1}(x)$ has precisely one zero in each of $2n-2$ intervals $(a_{j+1},b_j)$ and $(-b_j,-a_{j+1})$. 
Note that $\sgn \widetilde{D}_{2n+1}(b_n)=(-1)^{n+1}$ and $\sgn \widetilde{D}_{2n+1}(-b_n)=(-1)^{n}$.
It follows from~\eqref{dnxpnxqnx} that
$$\sgn \lim_{x \to 0}\frac{\widetilde{D}_{2n+1}(x)}{x}=(-1)^{n}.$$
Hence
$$\sgn \left(\lim_{x \to 0^-}\widetilde{D}_{2n+1}(x)\right)=(-1)^{n+1}, ~\sgn \left(\lim_{x \to 0^+}\widetilde{D}_{2n+1}(x)\right)=(-1)^{n}.$$
Therefore, $\widetilde{D}_{2n+1}(x)$ has precisely one zero in each of the intervals $(-b_n,0)$ and $(0,b_n)$. Moreover, $\widetilde{D}_{2n+1}(x)$ has a simple zero $x=0$.
Thus $\widetilde{D}_{2n+1}(x)\in\rz(-1,1)$.

In conclusion, we arrive at the real-rootedness of $\widetilde{D}_{2n}(x)$, i.e.,
$$\widetilde{D}_{2n}(x)=\frac{(2n)!}{2}\prod_{i=1}^n(x-c_i)(x+c_i),~\widetilde{D}_{2n+1}(x)=\frac{(2n+1)!}{2}x\prod_{i=1}^n(x-d_i)(x+d_i),$$
where $c_1>c_2>\cdots>c_{n-1}>c_n$ and $d_1>d_2>\cdots>d_{n-1}>d_n$.
Furthermore, we have
\begin{equation}\label{zeros-inequ1}
\begin{aligned}
&1>r_1>c_1>s_2>r_2>c_2>s_3>\cdots >r_{n-1}>c_{n-1}>s_n>r_n>c_n>0,\\
&1>b_1>d_1>a_2>b_2>d_2>a_3>\cdots >b_{n-1}>d_{n-1}>a_n>b_n>d_n>0.
\end{aligned}
\end{equation}
Then~\eqref{zeros01} and~\eqref{zeros02} are immediate from~\eqref{zeros-inequ1} and the transform $z=\frac{x-1}{x+1}$, as desired.
\end{proof}
\section*{Acknowledgements}
The author devoted more than a decade to the study of Brenti's suggestion~\cite[Section~5]{Brenti94} with the arguments of interlacing zeros.
During this period, the author 
benefited greatly from numerous discussions with my colleagues and fellow mathematicians, including Jun-Ying Liu and Chak-On Chow.
This work is supported by the National Natural Science Foundation of China (No. 12071063) and 
the Natural Science Foundation of Shandong Province of China (ZR2026MS0047).

\end{document}